\documentclass{article}
\usepackage[utf8]{inputenc}

\usepackage{amsmath}
\usepackage{amssymb}

\usepackage{url}

\usepackage{graphicx}
\graphicspath{ {Graphics/} }
\usepackage{wrapfig}
\usepackage[export]{adjustbox}
\usepackage{float}

\usepackage[font=small]{caption}

\usepackage{amsthm}
\newtheorem{theorem}{Theorem}
\newtheorem*{theorem*}{Main Theorem}
\newtheorem{lemma}{Lemma}
\newtheorem{proposition}{Proposition}
\newtheorem*{remark}{Remark}

\title{\textbf{Small exotic 4-manifolds from lines and quadrics in $\mathbb{CP}^{2}$}}
\author{Stefan Mihajlović}
\date{}

\begin{document}

\maketitle

\begin{abstract}
    We construct potentially new manifolds homeomorphic but not diffeomorphic to $\mathbb{CP}^{2} \# 8 \overline{\mathbb{CP}^{2}}$ and $\mathbb{CP}^{2} \# 9 \overline{\mathbb{CP}^{2}}$ via rational blowdown surgery along certain $4$-valent plumbing graphs. This way all the graph classes from \cite{weighted} have a representative which admits a rational blowdown leading to an exotic manifold. We emphasize the simplicity of the constructions which boils down to finding a good configuration of complex lines and quadrics in $\mathbb{CP}^{2}$, and deciding which intersections to blow up.
    
\end{abstract}

\section{Introduction}
\label{section_introduction}

Smooth 4-manifold topology is a very intriguing field which has been transformed by several techniques and constructions in the past decades. Constructing different smooth structures on any given smoothable $4$-manifold is still not a trivial problem, and for many of them it is not known whether there are different smooth structures, let alone if there is an infinite number of smoothings.

The problem we will be focusing on in this paper is the construction of \emph{small} exotic $4$-manifolds, meaning manifolds with small Euler characteristic and signature, homeomorphic but not diffeomorphic to some standard $4$-manifolds. Donaldson first proved that a certain smooth $4$-manifold admits two different smooth structures \cite{donaldson1987}, by using his newly constructed invariants to distinguish Dolgachev surfaces which are homeomorphic to $\mathbb{CP}^{2} \# 9 \overline{\mathbb{CP}^{2}}$. Since then there were several papers providing increasingly more intricate constructions of even smaller exotic manifolds \cite{Kotschick1989, Park2005, exotic6, exotic5, exotic3, exotic2and4}. In this note we prove the following:

\begin{theorem*}
There exists a configuration of complex lines and quadrics in $\mathbb{CP}^{2}$, and graphs from classes $\mathcal{B}^{4}$ and $\mathcal{C}^{4}$ shown in Figure \ref{grafovi}, which can be used to produce exotic $\mathbb{CP}^{2} \# 8 \overline{\mathbb{CP}^{2}}$ and $\mathbb{CP}^{2} \# 9 \overline{\mathbb{CP}^{2}}$ via rational blowdowns.
\end{theorem*}

Examples of non-standard smooth structures on these manifolds were already known \cite{donaldson1987, Kotschick1989}, as well as the general technique we are using - the \textit{rational blowdown surgery} introduced by Fintushel and Stern \cite{blowdown}. In its most general form, this surgery technique replaces an adequate embedded plumbing with some rational homology ball, simplifying the topology in a controlled way. In our considerations all plumbings are neighbourhoods of spheres pairwise intersecting transversely in at most one point, and the plumbing graph is a tree.

The novelty is using particular plumbings from two graph classes $\mathcal{B}^{4}$ and $\mathcal{C}^{4}$ from \cite{weighted} shown in Figure \ref{grafovi}, previously unknown to produce exotic manifolds via rational blowdown. This way we show that each class of graphs from \cite{weighted} has a representative which admits a rational blowdown leading to an exotic manifold, which might eventually advance the understanding of smoothings of singularities discussed there.

\begin{figure}[H]
    \includegraphics[width=1\textwidth]{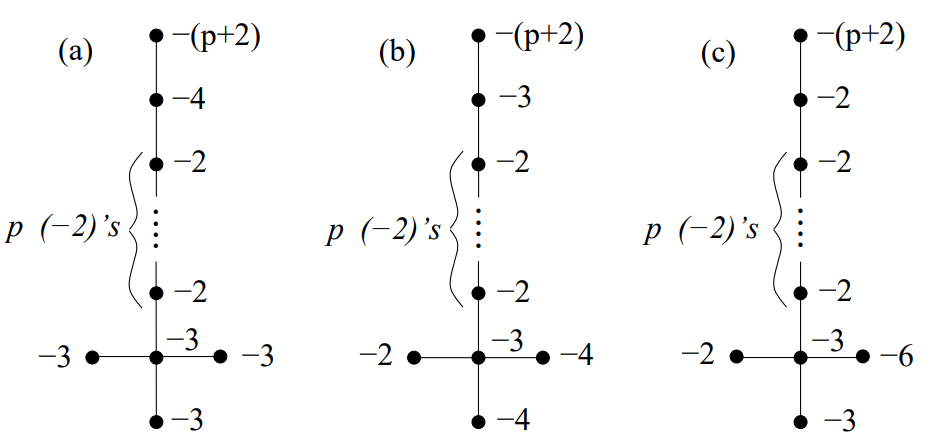}
    \centering
    \caption{Classes $\mathcal{A}^{4}$, $\mathcal{B}^{4}$ and $\mathcal{C}^{4}$}
    \label{grafovi}
\end{figure}

Here it is worth emphasizing that we are actually not looking at a pencil of curves, blowing it up, deforming the monodromies, and rationally blowing down. Rather, we start with a good configuration of degree $1$ and $2$ curves (complex lines and quadrics) in $\mathbb{CP}^{2}$ which are all already spheres by the genus-degree formula. Then we blow up some intersection points, and some additional generic points until we get a required configuration of intersecting spheres embedded in $\mathbb{CP}^{2}$ blown up some number of times. After rationally blowing down this configuration in a symplectic way, we determine the homeomorphism type and concisely show that the diffeomorphism type is not standard.

\vspace{0.2cm}

\noindent
\textbf{Acknowledgements:} I would like to thank my advisor Andr\'as Stipsicz for introducing me to smooth $4$-dimensional topology, pointing me to the problems discussed in this paper, and selflessly guiding me through my PhD journey.

\section{The curve configuration}
\label{section_configuration}

\noindent
The configuration of curves in $\mathbb{CP}^{2}$ that we start with is sketched in Figure \ref{curves} below. It will consist of two quadrics and four complex lines intersecting in a certain way, and it is derived by studying the configuration in the master thesis of Ta The Ahn \cite{master} where an example from class $\mathcal{A}^{4}$ was used in an exotic construction.

\begin{figure}[H]
    % command cuts the white space
    % \vspace{-4.0em}
    \includegraphics[width=1\textwidth]{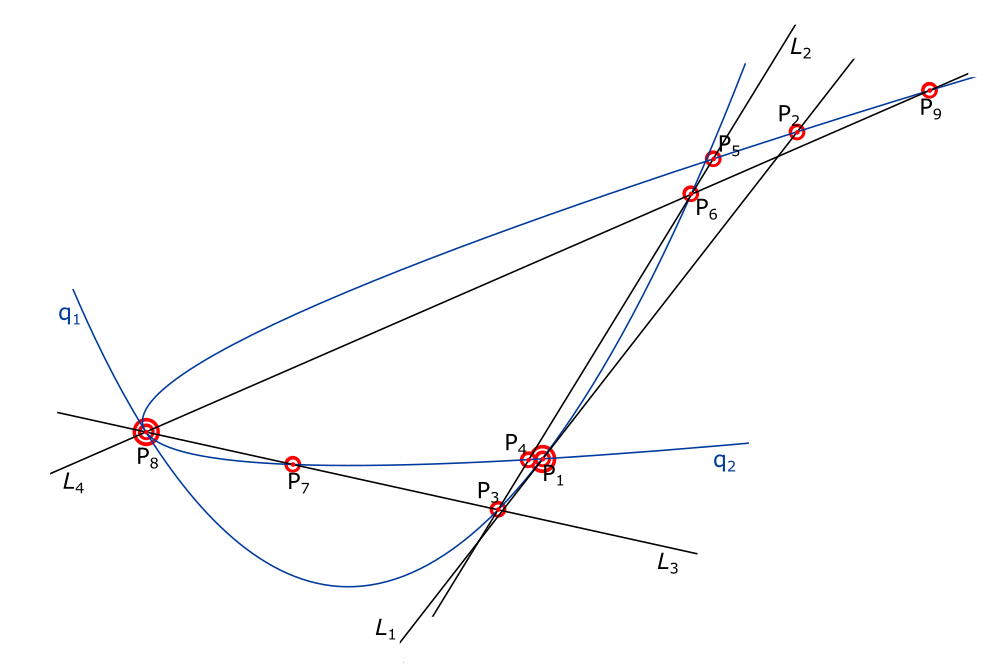}
    \centering
    \caption{Sketch of the curve configuration}
    \label{curves}
\end{figure}

First, take two irreducible quadrics $q_{1}$ and $q_{2}$ which are tangent at one point and have two more transverse intersections. We give an example of such two quadrics, defined in standard projective coordinates in $\mathbb{CP}^{2}$ by homogeneous degree $2$ equations:
$$ z_{1}^{2} + z_{2}^{2} + z_{3}^{2} = 0$$
$$ z_{1}z_{2} + 2\sqrt{2}i\cdot z_{2}z_{3} + z_{1}z_{3} = 0$$
Their common tangency is the point $[1 : \frac{\sqrt{2}}{2}i : \frac{\sqrt{2}}{2}i]$ which we further denote by $P_{8}$, and the two other intersections are $[-(1+\sqrt{3})\sqrt{2}i : -(2+\sqrt{3}) : 1]$ and $[-(1-\sqrt{3})\sqrt{2}i : -(2-\sqrt{3}) : 1]$. One general way to find two quadrics that intersect this way is by deforming equations of an irreducible quadric and a quadric consisting of a tangent to the irreducible one and a generic line.

After constructing $q_{1}$ and $q_{2}$, we take the tangent line to $q_{1}$ at one of the transverse intersection points with $q_{2}$, denote this point by $P_{1}$ and line by $L_{1}$. This tangent line intersects $q_{2}$ in another point, denote it $P_{2}$.
%Lines and quadrics intersect in two points which are generically different
Now take a generic line $L_{2}$ which intersects $q_{1}$ in points we name $P_{3}$ and $P_{6}$, and intersects $q_{2}$ in $P_{4}$ and $P_{5}$. Denote by $L_{3}$ the line passing through $P_{8}$ and $P_{3}$, and by $L_{4}$ the line going through $P_{8}$ and $P_{6}$. The other intersections of $L_{3}$ and $L_{4}$ with $q_{2}$ are denoted by $P_{7}$ and $P_{9}$ respectively.

\section{Blowing up and the incidence graph}
\label{section_incidencegraph}

We blow up $\mathbb{CP}^{2}$ as shown in Figure \ref{curves}, starting from the point $P_{1}$ to $P_{9}$. One red circle around a point means one blow up and two circles mean we did two consecutive blow ups completely removing the intersections at the points of tangency. Exceptional curves $e_{1}$ and $e_{2}$ correspond to the point $P_{1}$, $e_{3}$ corresponds to $P_{2}$, and so on, $e_{9}$ and $e_{10}$ correspond to $P_{8}$, and $e_{11}$ to $P_{9}$.

In the process of blowing up a point, any curve passing through this point can be transformed in a certain way (see e.g. \cite{4-manifolds, WildWorld}), and the result is called the \textit{proper transform} of the curve. One effect is that proper transforms of the curves which intersect transversely in the point that is blown up, no longer intersect in that point. Another is that the homology class of the proper transform is the homology class of the initial curve minus the class of the exceptional curve. In our example, after the initial $11$ blow ups, the homology classes of proper transforms of the curves and their self-intersections are as follows:

\begin{center}
\begin{tabular}{ | c | c | }
 \hline
    $\widetilde{q_{1}} = q_{1} - e_{1} - e_{2} - e_{4} - e_{7} - e_{9} - e_{10}$ & $\widetilde{q_{1}} \cdot \widetilde{q_{1}} = -2$ \\ 
    $\widetilde{q_{2}} = q_{2} - e_{1} - e_{3} - e_{5} - e_{6} - e_{8} - e_{9} - e_{10} - e_{11}$ & $\widetilde{q_{2}} \cdot \widetilde{q_{2}} = -4$ \\  
    $\widetilde{L_{1}} = L_{1} - e_{1} - e_{2} - e_{3}$ & $\widetilde{L_{1}} \cdot \widetilde{L_{1}} = -2$ \\
    $\widetilde{L_{2}} = L_{2} - e_{4} - e_{5} - e_{6} - e_{7}$ & $\widetilde{L_{2}} \cdot \widetilde{L_{2}} = -3$ \\
    $\widetilde{L_{3}} = L_{3} - e_{4} - e_{8} - e_{9}$ & $\widetilde{L_{3}} \cdot \widetilde{L_{3}} = -2$ \\
    $\widetilde{L_{4}} = L_{4} - e_{7} - e_{9} - e_{11}$ & $\widetilde{L_{4}} \cdot \widetilde{L_{4}} = -2$ \\
 \hline
 \end{tabular}
\end{center}

\begin{center}
\small{Table 1: Homology classes and self-intersections of curves after $11$ blow ups}
\end{center}

We can now form the incidence graph of the new configuration by representing curves as vertices, with an edge connecting vertices if there is an intersection between those two curves, as shown in Figure \ref{incidenceGraph}.

Two different ways of further blowing up intersection points in this configuration eventually give embedded plumbings from classes $\mathcal{B}^{4}$ and $\mathcal{C}^{4}$ of $4$-valent graphs from \cite{weighted}, and this is shown in the beginnings of the next two sections. Then we use the fact that these plumbings admit rational blowdown surgeries, and that they can be done in a symplectic way. Finally, we find the homeomorphism types of the resulting manifolds, and prove that they are exotic. The Main Theorem stated in the introduction is comprised of Theorem \ref{Th1} in section \ref{section_example1} and Theorem \ref{Th2} in section \ref{section_example2}.

\begin{figure}[H]
    \includegraphics[width=1\textwidth]{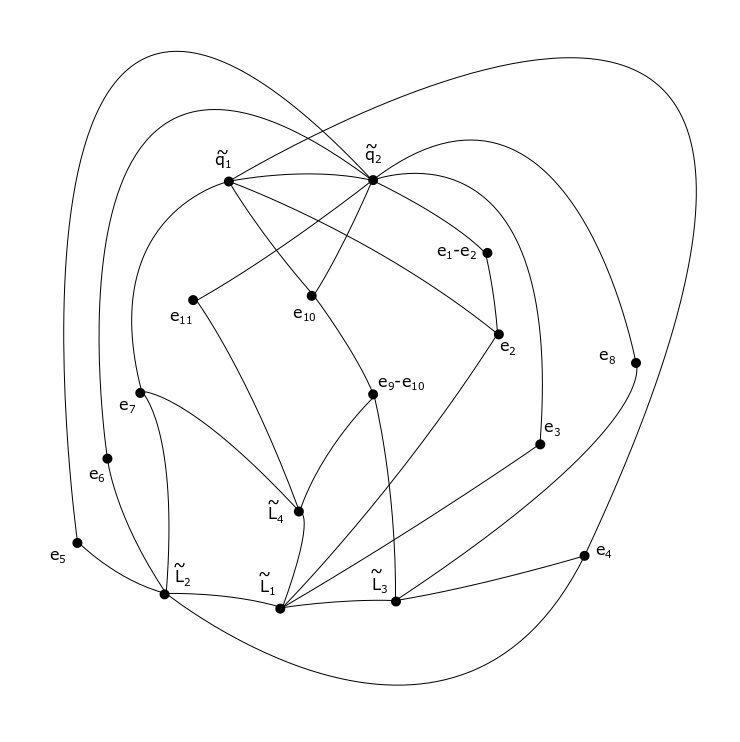}
    \centering
    \caption{The incidence graph of the curve configuration after $11$ blow ups}
    \label{incidenceGraph}
    % cuts white space
    %\vspace{-2.0em}
\end{figure}

\section{Exotic $\mathbb{CP}^{2} \# 8 \overline{\mathbb{CP}^{2}}$ via a graph from class $\mathcal{B}^{4}$}
\label{section_example1}

\vspace{0.2cm}

Start by Figure \ref{KlasaB} where we highlighted nodes and edges which will form the required subgraph. The homology classes of curves at this point are in Table 1. Blowing up the intersection of curves $\widetilde{q_{1}}$ and $e_{2}$, their self-intersections drop to $-3$ and $-2$, and we get a new exceptional sphere $e_{12}$. Doing the same with the intersection between $\widetilde{L_{2}}$ and $e_{4}$, their self-intersections drop to $-4$ and $-2$ and we get $e_{13}$. After three additional blow ups needed to achieve the self-intersections required for the rational blowdown surgery, we arrive to the subgraph shown in Figure \ref{graphB} which is of type $\mathcal{B}^{4}$ with $p=2$ using notation of Figure \ref{grafovi}: we can first blow up a generic point of $\widetilde{L_{1}}$, creating an exceptional curve $e_{14}$, and then two different generic points of $\widetilde{L_{4}}$, making two new exceptional curves $e_{15}$ and $e_{16}$.

\begin{figure}[H]
    \includegraphics[width=1\textwidth]{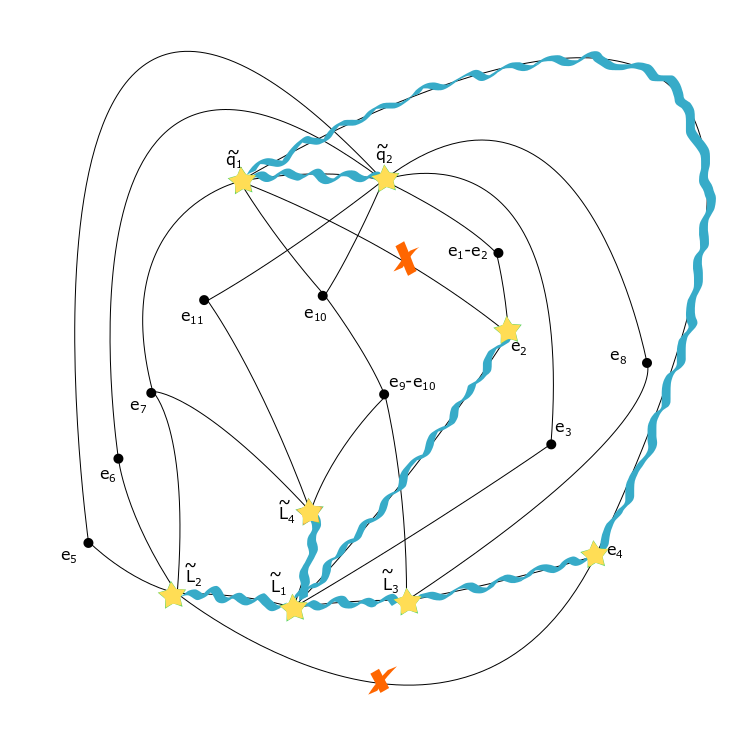}
    \centering
    \caption{Yellow stars are vertices and blue curly lines are edges which form a subgraph from class $\mathcal{B}^{4}$ presented in Figure \ref{graphB}. Orange X's show which $2$ intersections to blow up, whereas some additional blow ups used for adjusting the self-intersections to match the vertex markings in Figure \ref{graphB} are not visible here but described in the main text.}
    \label{KlasaB}
\end{figure}

Denote the final classes by $u_{1} = \widetilde{L_{2}}-e_{13}$, $u_{2} = \widetilde{L_{1}}-e_{14}$, $u_{3} = \widetilde{L_{4}}-e_{15}-e_{16}$, $u_{4} = e_{2}-e_{12}$, $u_{5} = \widetilde{L_{3}}$, $u_{6} = e_{4} - e_{13}$, $u_{7} = \widetilde{q_{1}}-e_{12}$ and $u_{8} = \widetilde{q_{2}}$. Therefore, after $16$ blow ups, we have the plumbing $P$ from Figure \ref{graphB} embedded in $\mathbb{CP}^{2} \# 16 \overline{\mathbb{CP}^{2}}$, and the homology classes of plumbing spheres are in Table 2:

\begin{center}
\begin{tabular}{|c|}
  \hline
  $u_{1} = h - e_{4} - e_{5} - e_{6} - e_{7} - e_{13}$ \\
  $u_{2} = h - e_{1} - e_{2} - e_{3} -e_{14}$   \\
  $u_{3} = h - e_{7} - e_{9} - e_{11} -e_{15}-e_{16}$ \\
  $u_{4} = e_{2}-e_{12}$ \\
  $u_{5} = h - e_{4} - e_{8} - e_{9}$ \\
  $u_{6} = e_{4} - e_{13}$ \\
  $u_{7} = 2h - e_{1} - e_{2} - e_{4} - e_{7} - e_{9} - e_{10} -e_{12}$ \\
  $u_{8} = 2h - e_{1} - e_{3} - e_{5} - e_{6} - e_{8} - e_{9} - e_{10} - e_{11}$ \\
  \hline
\end{tabular}
\end{center}

\begin{center}
\small{Table 2: Homology classes of spheres of the plumbing $P$}
\end{center}

\begin{figure}[H]
    \includegraphics[width=1\textwidth]{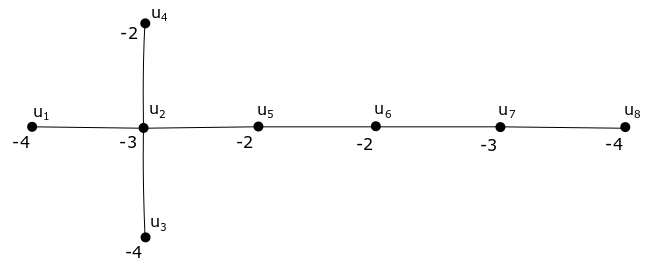}
    \centering
    \caption{Plumbing graph $P$ from class $\mathcal{B}^{4}$}
    \label{graphB}
\end{figure}

As our plumbing is from the class $\mathcal{B}^{4}$, by \cite[Theorem $1.6$]{weighted}, we can perform the rational blowdown along $P$ granting:
$$X = (\mathbb{CP}^{2} \# 16 \overline{\mathbb{CP}^{2}} - intP) \cup B$$
where $B$ is the rational homology ball smoothing of the normal surface singularity defined on pp. 1296-1297 of \cite{weighted} using results of \cite{rationalConstruct}.

An important point is that we can assume that the rational blowdown can be performed \textit{symplectically}, which follows from the main result of \cite{symplectic}. First, all the plumbing spheres of $P$ can be assumed to be symplectic submanifolds as proper transforms of complex submanifolds, and second, our plumbing graph is a negative definite tree \cite{weighted}. Then, from \cite[Theorem $1.1$]{symplectic}, the appropriate neighbourhood of the plumbing can be replaced by $B$ so that $(X, \omega_{X})$ is symplectic, and denoting $V = \mathbb{CP}^{2} \# 16 \overline{\mathbb{CP}^{2}} - intP$, there is a symplectomorphism $\phi_{V} : (V, \omega_{X}|_{V}) \longrightarrow (V, \omega|_{V})$, where $\omega$ is any symplectic structure on $\mathbb{CP}^{2} \# 16 \overline{\mathbb{CP}^{2}}$ that we started with.

Of course, this way we get a well-defined underlying smooth structure on the new manifold $X$. The main goal of this section is to prove the following:

\begin{theorem}
\label{Th1}
$X$ is homeomorphic but not diffeomorphic to $\mathbb{CP}^{2} \# 8 \overline{\mathbb{CP}^{2}}$.
\end{theorem}

\begin{proof}
Propositions \ref{propHomeo1} and \ref{propDifeo1} in upcoming subsections prove the theorem.
\end{proof}

\subsection{The topology of $X$}

To find the homeomorphism type of $X$, we use the foundational result of Freedman \cite{freedman}, which along with Donaldson's theorem \cite{donaldson1983} implies that :

\vspace{0.2cm}

\noindent
\textit{Two \textit{smooth} simply connected $4$-manifolds are homeomorphic if and only if their Euler characteristics, signatures, and parity of the intersection forms are equal.}

\vspace{0.2cm}

\noindent
First we need to prove that $X$ is simply connected, and to do so we will have three standard applications of Van Kampen's theorem. The main part is to prove that for the inclusion $i: \partial P \hookrightarrow \mathbb{CP}^{2} \# 16 \overline{\mathbb{CP}^{2}} - intP$, the homomorphism $i_{*}$ induced on fundamental groups is a trivial map.

From \cite[Theorem 5.1]{seifert}, the boundary $\partial P $ is a Seifert fibered 3-manifold with a Seifert ivariant $\{ 0; (1, 3), (2, 1), (4, 1), (4, 1), (25, 18) \}$. Its fundamental group is described by \cite[Theorem 6.1]{seifertLectures} which implies:

\begin{lemma}
\label{prviSeifert}
$\pi_{1}(\partial P)$ has a presentation given by generators $ q_{0}, q_{1}, q_{2}, q_{3}, q_{4}, h$ and relations:
\begin{itemize}
    \item $q_{0}q_{1}q_{2}q_{3}q_{4} = 1$
    \item $[h, q_{i}] = 1$ for all $ i = 0,1,2,3,4$
    \item $q_{0}h^{3} = 1$, $q_{1}^{2}h = 1$, $q_{2}^{4}h = 1$, $q_{3}^{4}h = 1$, $q_{4}^{25}h^{18} = 1$
\end{itemize}
Furthermore, the classes of $q_{1}, q_{2}$ and $q_{3}$ can be chosen to be normal circles to spheres $u_{4}$, $u_{1}$ and $u_{3}$ respectively.
\end{lemma}

\begin{lemma}
\label{trivial1}
$i_{*}(\pi_{1}(\partial P))$ is trivial.
\end{lemma}

\begin{proof}
We denoted $V = \mathbb{CP}^{2} \# 16 \overline{\mathbb{CP}^{2}} - intP$, meaning $V$ is the complement of the plumbing. The normal circle to the sphere $u_{3}$ can be contracted along the sphere which intersects it in a single point, and we can choose $e_{15}$ (or $e_{16}$) and contract that normal circle in $V$. Therefore, the corresponding generator trivializes through the inclusion, $i_{*}(q_{3}) = 1$. Relation $q_{3}^{4}h = 1$ from Lemma \ref{prviSeifert} gives $i_{*}(h) = 1$ and then $q_{0}h^{3} = 1$ implies $i_{*}(q_{0}) = 1$.

Looking at Figure \ref{KlasaB}, we can see that $\widetilde{L_{2}}$ and $\widetilde{L_{4}}$ do not intersect each other but intersect the sphere $e_{7}$ in one point each, and their proper transforms $u_{1}$ and $u_{3}$ do the same in the final picture. As $e_{7}$ is disjoint from the rest of the plumbing, normal circles to $u_{1}$ and $u_{3}$, namely $q_{2}$ and $q_{3}$, can be isotoped in $e_{7}$ to bound an annulus in $V$. Therefore, $i_{*}(q_{2}) = i_{*}(q_{3})$, so $i_{*}(q_{2}) = 1$ as well.

%The inclusions of $q_{1}$ and $q_{4}$ remain, so
From $q_{0}q_{1}q_{2}q_{3}q_{4} = 1$ we are left with $i_{*}(q_{1}q_{4})=1$, which we multiply by $i_{*}(q_{1})$ on the left. Using $i_{*}(q_{1})^{2}=1$ which holds since $q_{1}^{2}h =1$ and $i_{*}(h) = 1$, we get $i_{*}(q_{4})=i_{*}(q_{1})$. So we have $i_{*}(q_{4})^{2} = 1$ as well, but by deducing $i_{*}(q_{4})^{25}=1$ from the last relation in Lemma \ref{prviSeifert}, it follows that $i_{*}(q_{4})=1$.  Finally, $i_{*}(q_{1}) = i_{*}(q_{4}) = 1$ concludes the result.
\end{proof}

\begin{lemma}
\label{simple1}
$X$ is simply connected.
\end{lemma}
\begin{proof}
$X$ is constructed as the union of $V = \mathbb{CP}^{2} \# 16 \overline{\mathbb{CP}^{2}} - intP$ and some rational homology ball $B$ glued along $\partial P$. Therefore Van Kampen's theorem gives us a presentation of its fundamental group through fundamental groups of the two pieces.

To determine $\pi_{1}(V)$ we also apply Van Kampen's theorem, this time to the decomposition $\mathbb{CP}^{2} \# 16 \overline{\mathbb{CP}^{2}} = V \cup P$. The fundamental group of the plumbing $P$ is trivial because it is homotopic to a wedge sum of several spheres. Also, $\pi_{1}(\mathbb{CP}^{2} \# k \overline{\mathbb{CP}^{2}})$ is trivial for any $k$ because it can be built without $1$-handles, so from $\pi_{1}(\mathbb{CP}^{2} \# 16 \overline{\mathbb{CP}^{2}}) = \pi_{1}(V) *_{\pi_{1}(\partial P)} \pi_{1}(P)$ we get $1 = \pi_{1}(V) \big/ i_{*}(\pi_{1}(\partial P))$. Now Lemma \ref{trivial1} concludes that $\pi_{1}(V)$ is a trivial group.

We denote the inclusion of the boundary $\partial B$ into the rational homology ball $B$ by $j: \partial B \hookrightarrow B$, and $N := \langle i_{*}(x)\cdot j_{*}(x)^{-1} | x \in \pi_{1}(\partial B) \rangle$. From Van Kampen's theorem and the triviality of $\pi_{1}(V)$, we have that $\pi_{1}(X) = \pi_{1}(V) *_{N} \pi_{1}(B) = \pi_{1}(B) \big/ \langle j_{*}(x) | x \in \pi_{1}(\partial B) \rangle$. However, surjectivity of $j_{*}$ comes from the fact that our rational homology ball was constructed as a complement of a certain (dual) plumbing $P^{'}$ from $\mathbb{CP}^{2} \# k \overline{\mathbb{CP}^{2}}$ for some $k > 0$ (\cite[section 8.1]{rationalConstruct} and \cite[pp. 1296-1297]{weighted}). More precisely, from another application of Van Kampen's theorem on $\mathbb{CP}^{2} \# k \overline{\mathbb{CP}^{2}} = B \cup P^{'}$, we get $1 = \pi_{1}(B) \big/ \langle j_{*}(x) | x \in \pi_{1}(\partial B) \rangle$. Therefore, $X$ is simply connected.
\end{proof}

\begin{proposition}
\label{propHomeo1}
$X$ is homeomorphic to $\mathbb{CP}^{2} \# 8 \overline{\mathbb{CP}^{2}}$.
\end{proposition}
\begin{proof}
To calculate $\chi(X)$ and $\sigma(X)$ we use the formulas:
$$\chi(X) = \chi(\mathbb{CP}^{2} \# 16 \overline{\mathbb{CP}^{2}}) - \chi(P) + \chi(B) = 19 - 9 + 1 = 11$$
$$ \sigma(X) = \sigma(\mathbb{CP}^{2} \# 16 \overline{\mathbb{CP}^{2}}) - \sigma(P) + \sigma(B) = -15 - (-8) = -7$$
Rokhlin's theorem \cite{rokhlin} implies that if the signature of a smooth simply connected $4$-manifold is not divisible by $16$, its intersection form must be odd, so this is the case for $X$. Therefore, the three invariants of $X$ match the corresponding invariants of $\mathbb{CP}^{2} \# 8 \overline{\mathbb{CP}^{2}}$. As $X$ is simply connected by Lemma \ref{simple1}, it is homeomorphic to $\mathbb{CP}^{2} \# 8 \overline{\mathbb{CP}^{2}}$ as a consequence of Freedman's theorem.
\end{proof}

\subsection{Exoticness of $X$}

% could have used KODAIRA DIMENSION because it is a diffeo invariant but I would need MINIMALITY then
% a refference for this: survey https://arxiv.org/pdf/1511.04831.pdf

To prove that $X$ is not diffeomorphic to $\mathbb{CP}^{2} \# 8 \overline{\mathbb{CP}^{2}}$, we will use its symplectic structure $\omega_{X}$ explained earlier (coming from \cite{symplectic}), and the following result:

\begin{lemma}[\textbf{\cite{lili}, Theorem D}]
\label{ThD}
There is a unique symplectic structure on $\mathbb{CP}^{2} \# m \overline{\mathbb{CP}^{2}}$ for all $2 \leq m \leq 9$ up to diffeomorphism and deformation.
\end{lemma}

\begin{remark}
We will slightly abuse notation denoting symplectic forms as their cohomology classes. Poincar\' e dual of $\alpha$ will be denoted by $PD(\alpha)$.

\end{remark}

A symplectic structure $\Omega$ on a $4$-manifold $M$ determines a contractible family $\mathcal{J}$ of $\Omega$-compatible almost complex structures $J$ on the cotangent bundle $T^{*}M$. The first Chern class is the same for all $J \in \mathcal{J}$ and it is called the symplectic canonical class $K_{\Omega} = c_{1}(T^{*}M, J)$.

The strategy of proving that $X$ is exotic is as in \cite{Park2005}, to calculate the cup product of the symplectic class and a compatible canonical class on both $\mathbb{CP}^{2} \# 8 \overline{\mathbb{CP}^{2}}$ and $X$, see that the signs of these products differ, and prove that this is impossible because of the uniqueness result stated in Lemma \ref{ThD}.

Lemma \ref{stark} essentially stated as \cite[Lemma $5.4$]{starkston} presents a standard symplectic structure on $\mathbb{CP}^{2} \# k \overline{\mathbb{CP}^{2}}$ and calculates the sign of the required cup product to be negative. Lemma \ref{signLemma} shows that this product has to be negative for any symplectic structure on $\mathbb{CP}^{2} \# 8 \overline{\mathbb{CP}^{2}}$ or $\mathbb{CP}^{2} \# 9 \overline{\mathbb{CP}^{2}}$, and this is a rather special result for $\mathbb{CP}^{2} \# m \overline{\mathbb{CP}^{2}}$ given $2 \leq m \leq 9$. In general, the sign $K_{\omega} \cdot \omega$ can be used as a smooth invariant on a symplectic manifold only when we know the manifold in question is minimal, and this is called \textit{the Kodaira dimension}.

\begin{lemma}
\label{stark}
For every $k > 0$, $\mathbb{CP}^{2} \# k \overline{\mathbb{CP}^{2}}$ admits a symplectic structure $\omega$ that satisfies ${PD(\omega) = ah - b_{1}e_{1} -...-b_{k}e_{k}}$ for some positive rational numbers $a, b_{1},..., b_{k}$. For fixed $a > 0$, $b_{i}$'s can be chosen to be arbitrarily small. The induced canonical class $K := K_{\omega}$ satisfies $PD(K) = -3h + e_{1} +...+ e_{k}$ and for small enough $b_{i}'s$, we have $K \cdot \omega < 0$.
\end{lemma}

\begin{proof}
In $\mathbb{CP}^{2}$, the dual of the cohomology class of $\omega$ is $a h$ for some $a > 0$ and we can choose it to be rational - this is because the symplectic area of $\mathbb{CP}^{1} \subset \mathbb{CP}^{2}$ is a positive number $a$ and it can be normalized to be rational (we could normalize it so that $a = 1$, but keep "$a$" to see its importance). The proof of this lemma follows from \cite[section $7.1$]{mcduff}, and more precisely from Theorem $7.1.21$ on the existence and properties of the symplectic blow up. Namely, part (v) of that theorem implies that after the blow up, the cohomology class of the symplectic form changes as $\omega_{\Tilde{M}} = \omega_{M} - \pi \lambda^{2} PD(e)$. Here $e$ denotes the homology class of the exceptional curve and $\lambda$ is the radius of the ball removed in the process of the symplectic blow up as explained in \cite{mcduff}. Choosing the ball in Darboux's chart to be as small as needed and $\pi \lambda^{2}$ rational, and repeating the procedure $k$ times, gives us $PD(\omega) = ah - b_{1}e_{1} -...-b_{k}e_{k}$ as required.

Formula $(7.1.31)$ in \cite{mcduff} shows the canonical class of the blow up $\Tilde{M}$ to be $c_{1}(T^{*}\Tilde{M}) = c_{1}(T^{*}M) + PD(e)$. From the previous and $PD(K_{\mathbb{CP}^{2}}) = -3h$, we get $PD(K) = -3h + e_{1} +...+ e_{k}$. Finally, $K \cdot \omega = -3a + b_{1} +...+ b_{k}$ is negative for $b_{i}$'s small enough.
\end{proof}

\begin{lemma}
\label{signLemma}
For any symplectic structure $\overline{\omega}$ on $M = \mathbb{CP}^{2} \# m \overline{\mathbb{CP}^{2}}$ for $2 \leq m \leq 9$:
\begin{center}
    $K_{\overline{\omega}}\cdot \overline{\omega} < 0$\end{center}
\end{lemma}

\begin{proof}

This result essentially follows from Lemma \ref{ThD}, as $\overline{\omega}$ has to be deformation equivalent to the standard symplectic structure $\omega$, meaning that up to diffeomorphism, there is a path of symplectic forms on $M$ connecting them.

So there is a symplectomorphism $\psi: (M, \omega_{M}) \longrightarrow (\mathbb{CP}^{2} \# m \overline{\mathbb{CP}^{2}}, \omega)$ such that there is a path of symplectic forms $\omega_{t}$ connecting $\omega_{0} = \overline{\omega}$ and $\omega_{1} = \omega_{M}$. Naturality of Chern classes gives $K_{\omega_{M}} = \psi^{*}(K)$ so $K_{\omega_{M}} \cdot \omega_{M} = \psi^{*}(K) \cdot \psi^{*}(\omega) = \psi^{*}(K \cdot \omega) = K \cdot \omega$ so symplectomorphism does not change this product.

Assume that $K_{\overline{\omega}}\cdot \overline{\omega} \geq 0$.
Firstly, the canonical class $K_{\overline{\omega}}$ does not change by deformation so $PD(K_{\overline{\omega}}) = - 3h + e_{1} +...+ e_{m}$. Now $PD(\overline{\omega}) = a_{0}h + a_{1}e_{1} +...+ a_{m}e_{m}$ for some numbers $a_{i} \in \mathbb{R}$. However, as $\overline{\omega}$ is symplectic, we must have $\overline{\omega} \cdot \overline{\omega} > 0$ so $a_{0}^{2} > \sum_{i=1}^{m} a_{i}^{2}$. Having $K_{\overline{\omega}} \cdot \overline{\omega} = -3a_{0} - a_{1} -...- a_{m} \geq 0$, we get $3a_{0} \leq - (\sum_{i=1}^{m} a_{i})$. If $a_{0} \leq 0$, from the path of symplectic forms with $ PD(\omega_{t}) = a_{0}^{t} h + a_{1}^{t} e_{1} + ... + a_{m}^{t} e_{m}$, we would have a continuous funcition $a^{t}_{0}$ connecting $a_{0}^{0} = a_{0} \leq 0$ and $a_{0}^{1} > 0$ (as $a > 0$ for symplectomorphic $\omega$). Then there would be $\tau$ for which $a^{\tau}_{0} = 0$ and thus $\omega_{\tau} \cdot \omega_{\tau} \leq 0$, which is not possible. Therefore, $a_{0} > 0$ and from earlier we have $0 < 3a_{0} \leq - (\sum_{i=1}^{m} a_{i})$ so:

\vspace{0.1cm}
\begin{center}
    $9a_{0}^{2} \leq (\sum_{i=1}^{m} a_{i})^{2} \leq m (\sum_{i=1}^{m} a_{i}^{2}) \leq 9 (\sum_{i=1}^{m} a_{i}^{2}) < 9 a_{0}^{2}$
\end{center}
\vspace{0.1cm}

\noindent
provides the required contradiction using the Cauchy–Schwarz inequality.
\end{proof}

\begin{proposition}
\label{propDifeo1}
$X$ is not diffeomorphic to $\mathbb{CP}^{2} \# 8 \overline{\mathbb{CP}^{2}}$.
\end{proposition}

\begin{proof}

As mentioned, the strategy is to calculate the cup product of the symplectic class and a compatible canonical class for $X$, and see that the sign of this product is positive, which proves exoticness of $X$ using Lemma \ref{signLemma}.

Let $\omega$ denote the symplectic form on $\mathbb{CP}^{2} \# 16 \overline{\mathbb{CP}^{2}}$ provided by Lemma \ref{stark}, whose Poincar\' e dual is equal to:
\begin{center}
    $PD(\omega) = ah - b_{1}e_{1} -...-b_{16}e_{16}$
\end{center}

and let $K$ denote the corresponding canonical class:
\begin{center}
    $PD(K) = -3h + e_{1} +...+ e_{16}$
\end{center}

From the previous two we have:
\begin{center}
    $K \cdot \omega = -3a + b_{1} +...+ b_{16}$
\end{center}

The symplectic structure $\omega_{X}$ on $X$ obtained after the rational blow down, was defined earlier in section \ref{section_example1}, and it has a compatible symplectic canonical class $K_{X}$ coming from a generic almost complex structure compatible with $\omega_{X}$.

To be able to calculate $K_{X} \cdot \omega_{X}$, we will decompose the cohomology classes $K$ and $\omega$. Denoting again $V = \mathbb{CP}^{2} \# 16 \overline{\mathbb{CP}^{2}} - intP$, we have decompositions $\mathbb{CP}^{2} \# 16 \overline{\mathbb{CP}^{2}} = V \cup P$ and $X = V \cup B$.

As a first step, note that the boundary Seifered fibered $3$-manifold $\partial P = - \partial B$ is a \textit{rational homology sphere} because $\frac{3}{1} + \frac{1}{2} + \frac{1}{4} + \frac{1}{4} + \frac{18}{25} \neq 0$ (see section $1.2.3$ in \cite{saveliev}). To prove it directly, we can calculate $H_{1}(\partial P; \mathbb{Z})$ from Lemma \ref{prviSeifert} and see that it is a finite group, which then implies $H^{*}(\partial P; \mathbb{Q}) = H^{*}(S^{3}; \mathbb{Q})$.

From the Mayer-Vietoris sequences for decompositions $\mathbb{CP}^{2} \# 16 \overline{\mathbb{CP}^{2}} = V \cup P$ and $X = V \cup B$, we get exact sequences:
$$H^{1}(\partial P ; \mathbb{Q}) \longrightarrow H^{2}(\mathbb{CP}^{2} \# 16 \overline{\mathbb{CP}^{2}}; \mathbb{Q}) \longrightarrow H^{2}( V ; \mathbb{Q}) \oplus H^{2}( P ; \mathbb{Q}) \longrightarrow H^{2}(\partial P ; \mathbb{Q})$$
$$H^{1}(\partial B ; \mathbb{Q}) \longrightarrow H^{2}(X; \mathbb{Q}) \longrightarrow H^{2}( V ; \mathbb{Q}) \oplus H^{2}( B ; \mathbb{Q}) \longrightarrow H^{2}(\partial B ; \mathbb{Q})$$

\vspace{0.1cm}
The triviality in $\mathbb{Q}$-cohomology gives $H^{1}(\partial P ; \mathbb{Q}) = 0 = H^{2}(\partial P ; \mathbb{Q})$ and $H^{1}(\partial B ; \mathbb{Q}) = 0 = H^{2}(\partial B ; \mathbb{Q})$, so both middle arrows are isomorphisms. From the first sequence, we can decompose the cohomology classes:

\begin{center}
    $K = K|_{V} + K|_{P}$ and $\omega = \omega|_{V} + \omega|_{P}$
\end{center}

As $B$ is a rational homology $4$-ball, $H^{2}( B ; \mathbb{Q}) = 0$ so the second sequence gives that classes $K_{X}$ and $\omega_{X}$ satisfy:

\begin{center}
    $K_{X} = K_{X}|_{V} = \phi_{V}^{*}(K|_{V})$ and $\omega_{X} = \omega_{X}|_{V} = \phi_{V}^{*}(\omega|_{V})$
\end{center}

where $\phi_{V}$ is the symplectomorphism from the beginning of section $4$. So:

\begin{center}
    $K_{X} \cdot \omega_{X} =  \phi_{V}^{*}(K|_{V}) \cdot \phi_{V}^{*}(\omega|_{V}) = \phi_{V}^{*}(K|_{V} \cdot \omega|_{V}) = K|_{V} \cdot \omega|_{V} = K \cdot \omega - K|_{P} \cdot \omega|_{P} $
\end{center}

\begin{center}
\begin{tabular}{|c|}
     \hline
     $K_{X} \cdot \omega_{X} = K \cdot \omega - K|_{P} \cdot \omega|_{P}$ \\
     \hline
\end{tabular}
\end{center}

The intersection matrix $M$ of the plumbing $P$ is defined by the intersections $[ u_{i} \cdot u_{j} ]$ as in Figure \ref{graphB}:

\[
M =
\begin{bmatrix}

-4 & 1 \\

1 & -3 & 1 & 1 & 1 \\

 & 1 & -4 \\
 
 & 1 & & -2 \\
 
 & 1 & & & -2 & 1 \\
 
 & & & & 1 & -2 & 1 \\
 
 & & & & & 1 & -3 & 1 \\
 
 & & & & & & 1 & -4 \\

\end{bmatrix}
\]

Let $\{\gamma_{i}\}_{i=1}^{8}$ be the basis of $H^{2}(P; \mathbb{Q})$ which is dual to the basis $\{u_{i}\}_{i=1}^{8}$, meaning $\gamma_{i}(u_{j}) = \delta_{ij}$. Then the intersections $[\gamma_{i} \cdot \gamma_{j}]$ are given by $[M^{-1}]_{ij}$:

\[
M^{-1} = -\frac{1}{512} \cdot
\begin{bmatrix}

153 & 100 & 25 & 50 & 72 & 44 & 16 & 4 \\

100 & 400 & 100 & 200 & 288 & 176 & 64 & 16 \\

25 & 100 & 153 & 50 & 72 & 44 & 16 & 4 \\

50 & 200 & 50 & 356 & 144 & 88 & 32 & 8 \\

72 & 288 & 72 & 144 & 576 & 352 & 128 & 32 \\

44 & 176 & 44 & 88 & 352 & 528 & 192 & 48 \\

16 & 64 & 16 & 32 & 128 & 192 & 256 & 64 \\

4 & 16 & 4 & 8  & 32 & 48 & 64 & 144

\end{bmatrix}
\]

\vspace{0.1cm}

From $K|_{P} = \sum_{i=1}^{8} (K|_{P}(u_{i})) \gamma_{i}$, and $K|_{P}(u_{i}) = K(u_{i}) = PD(K) \cdot u_{i}$, we have $K|_{P} = \sum_{i=1}^{8} (PD(K) \cdot u_{i}) \gamma_{i}$. Taking the values of $u_{i}$'s from Table 2: $$K|_{P} = 2 \gamma_{1} + \gamma_{2} + 2 \gamma_{3} + \gamma_{7} + 2 \gamma_{8}$$

Analogously, we get $\omega|_{P} = \sum_{i=1}^{8} (PD(\omega) \cdot u_{i}) \gamma_{i}$:

\vspace{0.2cm}

$\omega|_{P} = (a - b_{4} - b_{5} - b_{6} - b_{7} - b_{13}) \gamma_{1} + (a - b_{1} - b_{2} - b_{3} - b_{14}) \gamma_{2} + (a - b_{7} - b_{9} - b_{11} - b_{15} - b_{16}) \gamma_{3} + (b_{2}-b_{12}) \gamma_{4} + (a - b_{4} - b_{8} - b_{9}) \gamma_{5} + (b_{4} - b_{13}) \gamma_{6} + (2a - b_{1} - b_{2} - b_{4} - b_{7} - b_{9} - b_{10} - b_{12}) \gamma_{7} + (2a - b_{1} - b_{3} - b_{5} - b_{6} - b_{8} - b_{9} - b_{10} - b_{11}) \gamma_{8}$

\vspace{0.2cm}

After calculating $K|_{P} \cdot \omega|_{P}$, we use $K_{X} \cdot \omega_{X} = K \cdot \omega - K|_{P} \cdot \omega|_{P}$ to get:

\vspace{0.2cm}

$K_{X} \cdot \omega_{X} = 5.625 a - 2.5b_{1} - 0.875 b_{2} -1.5 b_{3} -1.1875 b_{4} - 0.6875 b_{5} - 0.6875 b_{6}- 1.875 b_{7} - 1.25 b_{8} - 3.1875 b_{9} - 0.75 b_{10} - 0.6875 b_{11}- 0.875 b_{12} - 1.1875 b_{13} - 0.75 b_{14} + 0.0625 b_{15} + 0.0625 b_{16}$

\vspace{0.2cm}

We have $K_{X} \cdot \omega_{X} > 0$ because $a$ is positive and we can choose $b_{i}$'s to be arbitrarily small. If $X$ was diffeomorphic to $\mathbb{CP}^{2} \# 8 \overline{\mathbb{CP}^{2}}$, Lemma \ref{signLemma} would imply $K_{X} \cdot \omega_{X} < 0$ so this concludes that $X$ is exotic.
\end{proof}

\section{Exotic $\mathbb{CP}^{2} \# 9 \overline{\mathbb{CP}^{2}}$ via a graph from class $\mathcal{C}^{4}$}
\label{section_example2}

In this section we construct a different plumbing from the one in section \ref{section_example1}, again starting with the construction in section \ref{section_incidencegraph}. We keep the notation of some auxiliary objects as in the previous sections to simplify the exposition. Apart from the construction of the plumbing, all calculations are similar so we only emphasize the differences.

Starting from the incidence graph in Figure \ref{incidenceGraph}, in Figure \ref{KlasaC} we highlight nodes and edges which will form the required subgraph from $\mathcal{C}^{4}$.

\begin{figure}[H]
    \includegraphics[width=1\textwidth]{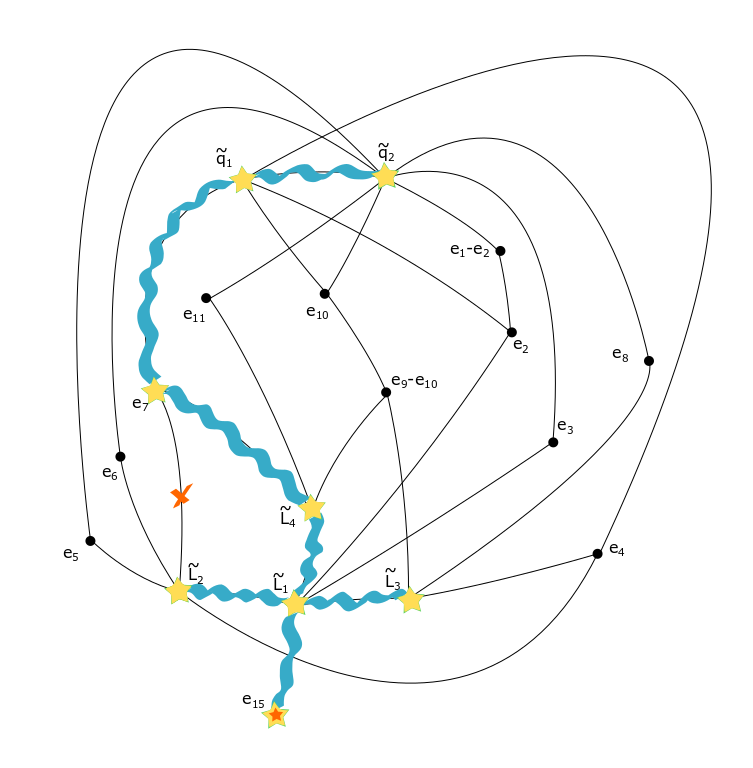}
    \centering
    \caption{Yellow stars are vertices and blue curly lines are edges which form a subgraph from class $\mathcal{C}^{4}$ presented in Figure \ref{graphC}. Note that $e_{15}$ is a new vertex compared to the starting Figure \ref{incidenceGraph}, marked with a smaller orange star because it comes from a new blow up. To arrive to an embedding, orange X shows which intersection to blow up. Some additional blow ups used for adjusting the self-intersections to match the vertex markings in Figure \ref{graphC} are not visible here but are described in the main text.}
    \label{KlasaC}
\end{figure}

We first blow up the intersection between $e_{7}$ and $\widetilde{L_{2}}$ and denote the exceptional curve by $e_{12}$. This way the proper transform of $\widetilde{L_{2}}$ gets self-intersection $-4$. With two further blow ups of different generic points of $\widetilde{L_{2}}$, we transform it into a curve of self-intersection $-6$, getting curves $e_{13}$ and $e_{14}$ in the process. Then blow up a generic point of the curve $\widetilde{L_{1}}$ getting $e_{15}$, and setting the self-intersection of the proper transform of $\widetilde{L_{1}}$ to $-3$. Now blow up a generic point of $e_{15}$, allowing its self-intersection to drop to $-2$, and name the exceptional curve $e_{16}$. Lastly, blow up a generic point of $\widetilde{L_{3}}$ dropping its self-intersection to $-3$ via the curve $e_{17}$.

Denote the classes by $v_{1} = e_{15}-e_{16}$, $v_{2} = \widetilde{L_{1}}-e_{15}$, $v_{3} = \widetilde{L_{3}} - e_{17}$, $v_{4} = \widetilde{L_{2}}-e_{12}-e_{13}-e_{14}$, $v_{5} = \widetilde{L_{4}}$, $v_{6} = e_{7} - e_{12}$, $v_{7} = \widetilde{q_{1}}$ and $v_{8} = \widetilde{q_{2}}$. These curves form the plumbing $Q$ embedded in $\mathbb{CP}^{2} \# 17 \overline{\mathbb{CP}^{2}}$, and its graph is presented in Figure \ref{graphC}. Therefore, the homology classes of spheres in the plumbing $Q$ are:

\begin{center}
\begin{tabular}{|c|}
  \hline
  $v_{1} = e_{15}-e_{16}$ \\
  $v_{2} = h - e_{1} - e_{2} - e_{3} - e_{15}$   \\
  $v_{3} = h - e_{4} - e_{8} - e_{9} - e_{17}$ \\
  $v_{4} = h - e_{4} - e_{5} - e_{6} - e_{7} -e_{12}-e_{13}-e_{14}$\\
  $v_{5} = h - e_{7} - e_{9} - e_{11}$ \\
  $v_{6} = e_{7} - e_{12}$ \\
  $v_{7} = 2h - e_{1} - e_{2} - e_{4} - e_{7} - e_{9} - e_{10}$ \\
  $v_{8} = 2h - e_{1} - e_{3} - e_{5} - e_{6} - e_{8} - e_{9} - e_{10} - e_{11}$ \\
  \hline
\end{tabular}
\end{center}

\begin{center}
\small{Table 3: Homology classes of spheres of the plumbing $Q$}
\end{center}

\begin{figure}[H]
    \includegraphics[width=1\textwidth]{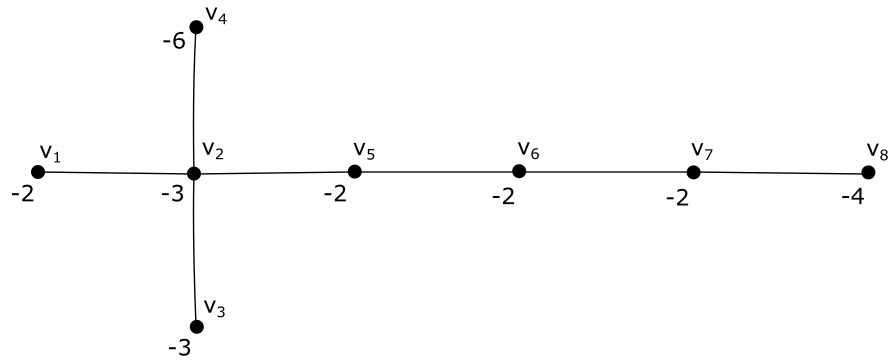}
    \centering
    \caption{Plumbing graph $Q$ from class $\mathcal{C}^{4}$}
    \label{graphC}
\end{figure}

\noindent
We can rationally blow down $Q$ by \cite{weighted} and get the manifold:

$$Y = (\mathbb{CP}^{2} \# 17 \overline{\mathbb{CP}^{2}} - intQ) \cup D$$

\vspace{0.1cm}

\noindent
where $D$ is a different rational homology ball than the one from section \ref{section_example1}. Details are very similar to the ones in the previous section and we only emphasize the differences, showing this time:

\begin{theorem}
\label{Th2}
$Y$ is homeomorphic but not diffeomorphic to $\mathbb{CP}^{2} \# 9 \overline{\mathbb{CP}^{2}}$.
\end{theorem}

\begin{proof}
Propositions \ref{propHomeo2} and \ref{propDifeo2} together will complete the proof.
\end{proof}

\subsection{The topology of $Y$}

In this example, the boundary $\partial Q $ is a Seifert fibered $3$-manifold \cite{seifert} with Seifert ivariant $\{ 0; (1, 3), (6, 1), (3, 1), (2, 1), (13, 10) \}$. Analagously to Lemma \ref{prviSeifert}, by \cite{seifertLectures} we have:

\begin{lemma}
\label{drugiSeifert}
$\pi_{1}(\partial Q)$ has a presentation given by generators $q_{0}, q_{1}, q_{2}, q_{3}, q_{4}, h$ and relations:
\begin{itemize}
    \item $q_{0}q_{1}q_{2}q_{3}q_{4} = 1$
    \item $[h, q_{i}] = 1$ for all $ i = 0,1,2,3,4$
    \item $q_{0}h^{3} = 1$, $q_{1}^{6}h = 1$, $q_{2}^{3}h = 1$, $q_{3}^{2}h = 1$, $q_{4}^{13}h^{10} = 1$
\end{itemize}
Furthermore, the classes of $q_{1}, q_{2}$ and $q_{3}$ can be chosen to be normal circles to spheres $v_{4}$, $v_{3}$ and $v_{1}$ respectively.
\end{lemma}

\begin{lemma}
\label{trivial2}
$i_{*}(\pi_{1}(\partial Q))$ is trivial.
\end{lemma}

\begin{proof}
In this case, compared to the previous section, it is easier to deduce the triviality of $i_{*}(\pi_{1}(\partial Q))$, as we made a lot of generic blow ups. More precisely, each of the three leaves of the plumbing graph $Q$ in Figure \ref{graphC}, that is $v_{4}$, $v_{3}$ and $v_{1}$, is intersecting a different exceptional sphere otherwise disjoint from the plumbing. As in the proof of Lemma \ref{trivial1}, the normal circles can be contracted in the complement of $Q$, so we can deduce $i_{*}(q_{1}) = 1$, $i_{*}(q_{2}) = 1$ and $i_{*}(q_{3}) = 1$. From $q_{1}^{6}h = 1$, we get $i_{*}(h) = 1$ and then $q_{0}h^{3} = 1$ implies $i_{*}(q_{0}) = 1$. The first relation of Lemma \ref{drugiSeifert} now gives $i_{*}(q_{4}) = 1$ and concludes that $i_{*}(\pi_{1}(\partial Q))$ is a trivial group.
\end{proof}

\begin{lemma}
\label{simple2}
$Y$ is simply connected.
\end{lemma}

\begin{proof}
Using Lemma \ref{trivial2} instead of Lemma \ref{trivial1}, the proof is analogous to the proof of Lemma \ref{simple1}.
\end{proof}

\begin{proposition}
\label{propHomeo2}
$Y$ is homeomorphic to $\mathbb{CP}^{2} \# 9 \overline{\mathbb{CP}^{2}}$.
\end{proposition}
\begin{proof}
As before we have:
$$\chi(Y) = \chi(\mathbb{CP}^{2} \# 17 \overline{\mathbb{CP}^{2}}) - \chi(Q) + \chi(D) = 20 - 9 + 1 = 12$$
$$ \sigma(Y) = \sigma(\mathbb{CP}^{2} \# 17 \overline{\mathbb{CP}^{2}}) - \sigma(Q) + \sigma(D) = -16 - (-8) = -8$$
$Y$ has an odd intersection form by Rohlkin's theorem \cite{rokhlin} and thus, all the invariants match the ones of $\mathbb{CP}^{2} \# 9 \overline{\mathbb{CP}^{2}}$. From Lemma \ref{simple2}, these $4$-manifolds are both simply connected, so by Freedman's theorem we get that they must be homeomorphic.
\end{proof}

\subsection{Exoticness of $Y$}

\begin{proposition}
\label{propDifeo2}
$Y$ is not diffeomorphic to $\mathbb{CP}^{2} \# 9 \overline{\mathbb{CP}^{2}}$.
\end{proposition}

\begin{proof}

The proof is essentially the same as the proof of Proposition \ref{propDifeo1}. Start by introducing a symplectic form on $\mathbb{CP}^{2} \# 17 \overline{\mathbb{CP}^{2}}$ using Lemma \ref{stark}:

\begin{center}
    $ PD(\omega) = ah - b_{1}e_{1} -...-b_{17}e_{17}$
\end{center}

This time, let $K$ be the standard canonical class of $\mathbb{CP}^{2} \# 17 \overline{\mathbb{CP}^{2}}$:

\begin{center}
    $PD(K) = -3h + e_{1} +...+ e_{17}$
\end{center}

From these two we have:
\begin{center}
    $K \cdot \omega = -3a + b_{1} +...+ b_{17}$
\end{center}

\vspace{0.1cm}

%maybe change matrix bracket style to "("?
The intersection matrix of the plumbing $Q$ is $[v_{i} \cdot v_{j}]$:

\[
N =
\begin{bmatrix}

-2 & 1 \\

1 & -3 & 1 & 1 & 1 \\

 & 1 & -3 \\
 
 & 1 & & -6 \\
 
 & 1 & & & -2 & 1 \\
 
 & & & & 1 & -2 & 1 \\
 
 & & & & & 1 & -2 & 1 \\
 
 & & & & & & 1 & -4 \\

\end{bmatrix}
\]

\vspace{0.2cm}

The intersection matrix of the basis $\{\gamma_{i}\}_{i=1}^{8}$ dual to $\{v_{i}\}_{i=1}^{8}$ is:

\[
N^{-1} = -\frac{1}{576} \cdot
\begin{bmatrix}

405 & 234 & 78 & 39 & 180 & 126 & 72 & 18 \\

234 & 468 & 156 & 78 & 360 & 252 & 144 & 36 \\

78 & 156 & 244 & 26 & 120 & 84 & 48 & 12 \\

39 & 78 & 26 & 109 & 60 & 42 & 24 & 6 \\

180 & 360 & 120 & 60 & 720 & 504 & 288 & 72 \\

126 & 252 & 84 & 42 & 504 & 756 & 432 & 108 \\

72 & 144 & 48 & 24 & 288 & 432 & 576 & 144 \\

18 & 36 & 12 & 6 & 72 & 108 & 144 & 180

\end{bmatrix}
\]

\vspace{0.2cm}

To calculate $K_{Y} \cdot \omega_{Y}$, we can aquire $K|_{Q}$ and $\omega|_{Q}$ decomposing the second cohomology classes as before. Again, this is possible because the boundary manifold $\partial Q$ is Seifert fibered and $\frac{3}{1} + \frac{1}{6} + \frac{1}{3} + \frac{1}{2} + \frac{10}{13} \neq 0$, so it is a rational homology sphere (see \cite{saveliev}). $K|_{Q} = \sum_{i=1}^{8} (PD(K) \cdot v_{i}) \gamma_{i}$ so using the values of $PD(K)$ and $v_{i}$'s from Table 3:

\begin{center}
$K|_{Q} = \gamma_{2} + \gamma_{3} + 4 \gamma_{4} + 2 \gamma_{8}$
\end{center}

A similar formula $\omega|_{Q} = \sum_{i=1}^{8} (PD(\omega) \cdot v_{i}) \gamma_{i}$ gives: 

\vspace{0.2cm}

$\omega|_{Q} = (b_{15} - b_{16}) \gamma_{1} + (a - b_{1} - b_{2} - b_{3} - b_{15}) \gamma_{2} + (a - b_{4} - b_{8} - b_{9} - b_{17}) \gamma_{3} + (a - b_{4} - b_{5} - b_{6} - b_{7} - b_{12} - b_{13} - b_{14}) \gamma_{4} + (a - b_{7} - b_{9} - b_{11}) \gamma_{5} + (b_{7} - b_{12}) \gamma_{6} + (2a - b_{1} - b_{2} - b_{4} - b_{7} - b_{9} - b_{10}) \gamma_{7} + (2a - b_{1} - b_{3} - b_{5} - b_{6} - b_{8} - b_{9} - b_{10} - b_{11}) \gamma_{8}$

\vspace{0.2cm}

And once again, from $K_{Y} \cdot \omega_{Y} = K \cdot \omega - K|_{Q} \cdot \omega|_{Q}$:

\vspace{0.2cm}

$K_{Y} \cdot \omega_{Y} = 5.625 a - 2.5 b_{1} - 1.75 b_{2} - 1.5 b_{3} -1.875 b_{4} - 0.708\overline{3} b_{5} - 0.708\overline{3} b_{6}- 1.208\overline{3} b_{7} - 0.\overline{6} b_{8} - 3.1\overline{6} b_{9} - 0.69791\overline{6} b_{10} - 1.25 b_{11} - 1.208\overline{3} b_{12} + 0.041\overline{6} b_{13} + 0.041\overline{6} b_{14} + 0.125 b_{15} + 0.125 b_{16} + 0.08\overline{3} b_{17}$

\vspace{0.2cm}

$K_{Y} \cdot \omega_{Y} > 0$ because $a$ is positive and $b_{i}$'s can be arbitrarily small. By Lemma \ref{signLemma}, this is impossible unless $Y$ is exotic.
\end{proof}

\begin{remark}
Finding interesting configurations of lines and quadrics could produce even smaller exotic $4$-manifolds via suitable rational blowdowns, so this is one upcoming challenge. It seems that the exoticness proof will remain true if enough curves from the initial configuration are used in the plumbing, so it would only remain to take care of simple connectedness.
\end{remark}

%\printbibliography
\bibliographystyle{abbrv}
\bibliography{sample}

\end{document}